%% file: damp.tex
\documentclass[twoside, final]{amsart} 
\usepackage{amsmath, amsfonts, amsthm, amssymb}
\usepackage{amscd} 
\usepackage[greek,english]{babel} 
\usepackage{verbatim} 
\usepackage{exscale}
\usepackage{cite}
\usepackage{epsfig}
\usepackage{amscd}
\usepackage{graphics}
\usepackage{color}

\usepackage{bbold}

\usepackage[notref, notcite]{showkeys}

\renewcommand{\geq}{\geqslant}
\renewcommand{\leq}{\leqslant}
\newcommand{\bs}{\boldsymbol}

\newtheorem{lemma}{Lemma}[section] 
 
\newtheorem{thm}[lemma]{Theorem} 
\newtheorem{cor}[lemma]{Corollary} 
 
\newtheorem*{thm*}{Theorem} 
\newtheorem*{prop*}{Proposition} 
\newtheorem*{cor*}{Corollary} 
\newtheorem*{conj*}{Conjecture} 
\numberwithin{equation}{section} 
\theoremstyle{remark} 
\newtheorem{rem}[lemma]{Remark} 
\newtheorem*{rem*}{Remark} 
\theoremstyle{definition} 
 
\newtheorem*{Def*}{Definition}

\def\ZZ{{\mathbb Z}}
\def\NN{{\mathbb N}}
\def\RR{{\mathbb R}}
\def\CC{{\mathbb C}}
\def\SS{{\mathbb S}}

\newcommand{\cB}{{\mathcal B}}

\newcommand{\cH}{\mathcal H}

\newcommand{\cL}{{\mathcal L}}

\newcommand{\cN}{{\mathcal N}}
\newcommand{\cO}{{\mathcal O}}

\newcommand{\cS}{\mathcal S}

\def\tP{\widetilde{P}}

\newcommand{\sH}{\mathsf{H}}

\def\be{\begin{eqnarray*}}
\def\ee{\end{eqnarray*}}
\def\ben{\begin{eqnarray}}
\def\een{\end{eqnarray}}

\newcommand{\Lap}{\Delta}
\newcommand{\abs}[1]{{\left\lvert{#1}\right\rvert}}
\newcommand{\norm}[1]{{\left\lVert{#1}\right\rVert}}
\newcommand{\ang}[1]{{\left\langle{#1}\right\rangle}}
\renewcommand{\d}{\partial}

\newcommand{\ep}{{\epsilon}}
\newcommand{\hamvf}{{\textsf{H}}}

\newcommand{\WF}{\operatorname{WF}} 
\newcommand{\supp}{\operatorname {supp}}
\newcommand{\e}{\operatorname{e}}
\newcommand{\WFh}{\WF_{\semi}} 
\newcommand{\semi}{\hbar} 
\renewcommand{\Im }{\operatorname{Im}}
\renewcommand{\Re}{\operatorname{Re}}
\renewcommand{\Pr }{\operatorname{Pr}}
\newcommand{\im }{\operatorname{Im}}
\newcommand{\re}{\operatorname{Re}}
\newcommand{\Id}{\operatorname{Id}} 
\renewcommand{\i}{\operatorname{i}} 
\def\phi{\varphi}

\newcommand{\CI}{C^\infty}

\begin{document}

\title{From resolvent estimates to damped waves} 
\date{\today} 
\author{Hans Christianson}
\author{Emmanuel Schenck}
\author{Andr\'as Vasy}
\author{Jared Wunsch}
\address{Department of Mathematics, University of North Carolina}
\address{D\'epartement de Math\'ematiques, Universit\'e Paris 13}
\address{Department of Mathematics, Stanford University}
\address{Department of Mathematics, Northwestern University}
\email{hans@unc.edu}
\email{schenck@math.univ-paris13.fr}
\email{andras@math.stanford.edu}
\email{jwunsch@math.northwestern.edu}

 \thanks{The authors are grateful for partial support from NSF grants
   DMS-0900524 (HC), DMS-1068742 (AV), DMS-1001463 (JW)}

\begin{abstract} 
  In this paper we show how to obtain decay estimates for the damped
  wave equation on a compact manifold without geometric control via
  knowledge of the dynamics near the un-damped set.  We show that if
  replacing the damping term with a higher-order \emph{complex
    absorbing potential} gives an operator enjoying polynomial
  resolvent bounds on the real axis, then the ``resolvent'' associated
  to our damped problem enjoys bounds of the same order.  It is known
  that the necessary estimates with complex absorbing potential can
  also be obtained via gluing from estimates for corresponding
  non-compact models.
\end{abstract}

\maketitle 

\section{Introduction}\label{S:intro}

On a compact, connected Riemannian manifold without
boundary $(X,g)$, we consider the non-selfadjoint Schr\"odinger operator
\begin{equation}\label{e:nsa op}
P(h)=h^2\Delta_g+\i h a\,
\end{equation}
where $a\in C^\infty (X)$ is a non-negative function, and 
$\Delta_g=d^*d$ is the \emph{non-negative} Laplacian associated to the
metric $g.$
This paper mainly addresses the question of the semiclassical analysis of the resolvent of $P(h)$, 
$$
R_z(h):=(P(h)-z)^{-1}
$$
for $z$ in a complex $h-$dependent neighborhood of 1.  
For non-selfadjoint
operators, it is well known that the norm of the resolvent
$\|R_z(h)\|_{\cL(L^2,L^2)}$ may be large, even far from the spectrum
\cite{EmTr05}, and a better understanding of the resolvent properties of
non-selfadjoint operators remains a challenging problem \cite{Sjo09}.  In this
paper we are particularly interested in (polynomial) upper bounds in $h$
for the resolvent. These bounds are especially useful when studying the
stabilization problem, which deals with the rate of the energy decay of the
solution of the \emph{damped wave equation} on $X$ :
\begin{equation}
\label{eq:dwe}
\left\{ \begin{array}{l}
\left( \partial_t^2 + \Delta_g + 
a(x) \partial_t \right) u(x,t) = 0, \quad (x,t) \in X \times (0, \infty) \\
u(x,0) = u_0\in H^1(X), \quad \partial_t u(x,0) = u_1\in H^0(X).
\end{array} \right.
\end{equation}
It has been shown (see \cite{Leb93}) that if $a >0$ somewhere, then the energy of the waves,
$$
E(u,t)=\frac{1}{2}\int_X | \partial_t u |^2 + | \nabla u |^2 dx 
$$
satisfies $E(u,t)\xrightarrow{t \to\infty }0$ for any initial data
$(u_0,u_1)\in H^1\times H^0$. If some monotone decreasing function $f(t)$  can be found such that 
$$E(u,t)\leq f(t)E(u,0)\,,
$$
so-called strong stabilization occurs. It is not hard to show that this is
equivalent to a uniform exponential decay : $\exists C,\beta>0$ such
that for
any $u$ solution of \eqref{eq:dwe},
$$E(u,t)\leq Ce^{-\beta t}E(u,0).$$
In pioneering works of Rauch, Taylor, Bardos and Lebeau \cite{RaTa75,
  BaLeRa92, Leb93}, it has been shown in various settings that strong
stabilization is equivalent to the geometric control condition (GCC) :
there exists $T_0>0$ such that from every point in
$\Sigma=\{|\xi|^2_g=1\}\subset T^*X$, the bicharacteristic of $P(h)$
reaches $\{a>0\}$ in time $\leq T_0$.  By contrast, when the manifold $X$ is no longer
controlled by $a$, decay rate estimates usually involve additional
regularity of the initial data. They take the form
$$E(u,t)\leq f_s(t)\|u\|^2_{\cH^s}$$
 for $s>0$ and
$$
\|u\|_{\cH^s}=\|u(0)\|^2_{H^{1+s}}+\|\partial_t u(0)\|^2_{H^s}\,.
$$
The question of exponential energy decay reduces to the study of
high-frequency phenomena, in particular the issue of the spectral
properties in the semiclassical limit $h\to 0$ of certain
non-selfadjoint operators approximately of the form\footnote{Strictly
  speaking, in order to apply resolvent bounds to the damped
  wave equation, we also need the imaginary part of the Schr\"odinger
  operator to be mildly $z$ dependent, with
\[
R_z(h) = (h^2\Delta_g + \i
ha \sqrt{z} -z)^{-1};
\]
this will be handled by perturbation (see Corollary
\ref{C:bb-1},
Section
\ref{s:dwe}, as well as references \cite{Leb93, Hit03, Chr-wave-2}).}
 \eqref{e:nsa op}, on a fixed energy layer. For
instance, when geometric control holds, there exist $h_0>0$, $C,
c>0$ such that for $z\in [1-\delta,1+\delta]+\i [-ch,ch]$ we have
\begin{equation}\label{eq:damp-glob-bd} 
\|R_z(h)\|_{\cL(L^2,L^2)}\leq C/h,\ h<h_0. 
\end{equation} 
Standard arguments then show that this resolvent estimate implies the
uniform exponential decay for the energy.  
Similar arguments will apply in the case considered here, of resolvent estimates with
loss.

\subsection{Motivation}
While our main motivation for studying resolvent estimates for $P(h)$ come
from the stabilization problem, our approach in this paper is oriented by
geometric considerations, as we explain now. As discussed above, in the
presence of geometric control, the resolvent $(P(h)-z)^{-1}$ enjoys a
polynomial bound in a neighbourhood of size $ch$ around the real axis, and this
property implies exponential damping. When geometric control no longer holds,
it is then a natural question to ask what type of estimate can be satisfied
by $\|(P(h)-z)^{-1}\|$, and, crucially, in what type of
complex neighbourhood of the real axis a resolvent estimate can be
obtained. The properties of the undamped set
\begin{equation}
\cN=\{\rho\in S^*X: \forall t\in\mathbb R, a\circ \e^{t\sH_p}(\rho)=0 \}
\end{equation}
are of central importance for this question.  Here,
$\sH_p$ denotes the Hamiltonian vector field generated by the principal symbol
$p=\sigma_h(P(h))=\lvert\xi\rvert_g^2$ of the operator $P(h),$ and
$S^*X=p^{-1}(1)$ denotes the unit cosphere bundle.  We remind the reader
that the flow generated by $\sH_p$ in $S^*X$ is simply the geodesic flow.

We now review some known results in the case $\cN\neq
\emptyset$. In \cite{Chr-NC-erratum}, the case when $\cN$ is a single hyperbolic
orbit is analyzed, and a polynomial resolvent estimate for $\|(P-z)^{-1}\|$
is shown in a $h/|\log h|$-size neighbourhood of the real axis. As a
consequence, the energy decay is sub-exponential : with the above
notation, one can get $f_s(t)=\e^{-\beta_s \sqrt t}$. It is known from
recent work \cite{BuCh-dw} that this decay is sharp. 
If the curvature of $X$ is assumed to be negative, and if the relative
size of the damping function $a$ is sufficiently large, 
then the resolvent
obeys a polynomial bound in a size $h$ neighbourhood of the real axis, and
as a result, exponential decay for regular initial data occurs \cite
{Sch10, Sch11}. Indeed, the hypotheses in \cite{Sch11} is much more general,
requiring only undamped sets of small pressure; the hyperbolic geodesic is a special
case. We note also that for \emph{constant} negative curvature, the
need for an arbitrarily 
large damping function $a$
has been recently removed by
Nonnenmacher, by using different methods \cite{Non11}.

A natural question raised by the above remarks, is the following: to what
extent does the geometry of the trapped set \emph{alone} determine a type of
decay? In other words, given a trapped geometry, what type of resolvent
estimate do we expect, and in what complex neighbourhood of the real axis?
This amounts in many cases to a potentially rather crude decay rate for the energy,
as it only depends on the structure of $\cN$ and not on the global
dynamics of geodesics passing through the damping; in certain cases,
however, our results can be
seen to be sharp.\footnote{A very natural further question
  would then be : given a trapped geometry, what kind of global assumptions
  on the manifold can improve the crude decay rate obtained when only $\cN$
  is known?}  

Motivated by the ``black box''
approach of Burq-Zworski \cite{BuZw04} (cf.\ earlier work of
Sj\"ostrand-Zworski \cite{SjZw91}) as well as recent work on the
gluing of resolvent estimates by Datchev-Vasy
\cite{Datchev-Vasy:Gluing}, we give a recipe for taking
information from \emph{resolvent} estimates obtained for a noncompact
problem in which the set $\cN$ consists of all trapped
geodesics---those not escaping to infinity---and investigate what
these estimates imply for the compact problem with damping.  In
practice, as recent results of Datchev-Vasy \cite{Datchev-Vasy:Gluing}
have shown the resolvent estimates on manifolds with, say,
asymptotically Euclidean ends to be equivalent to estimates on a
compact manifold with a \emph{complex absorbing potential}
substituting for the noncompact ends, we choose for the sake of
brevity and elegance to take this latter model as our ``noncompact''
setting.  As will be discussed below, these complex absorbing
potentials have the effect of annihilating semiclassical wavefront set
along geodesics passing through them; this is why they are roughly
interchangeable with noncompact ends, along which energy can flow off to
infinity never to return.

We thus formulate our main question as follows.  Assuming that
$(X,g)$ and $a$ are given, we consider a model operator of the form
$$
P_1(h)=h^2\Delta_g+\i  W_1
$$
in which the damping is replaced by a \emph{complex absorbing potential.}
We assume that this model operator enjoys a given ``resolvent''
estimate\footnote{We refer to this as a resolvent estimate owing to its
  close relationship with the estimate for the resolvent in scattering
  problems.}  on the real axis:
\begin{equation}\label{intro:resolvent}
\| (P_1(h)-z)^{-1}  \| \leq C \frac{\alpha(h)}{h},\quad z \in [1-\delta,1+\delta].
\end{equation}
Various examples of such
estimates already appear in the literature, see for instance
\cite{VaZw, CPV, Chr-disp-1, Chr-NC-erratum, Chr-surf-res, NoZw09_1, WuZw, ChWu-lsm} and the references
therein. 
Given \eqref{intro:resolvent}
we then aim to obtain analogous estimates for the inverse of the
operator with damping, i.e., on
$$
(P(h)-z)^{-1}=(h^2\Delta_g+\i h 
a-z)^{-1}\,,
$$ when the complex absorbing potential has crucially been
replaced by an $O(h)$ damping term. In this paper, we address this
question using a control theory argument motivated by \cite{BuZw04},
together with a recently improved estimate on resolvents truncated away from the
trapped set on one side 
\cite{DaVa-trapping-2}, and show that we
obtain the same order of estimate as for the model operator.  In the
next subsection we state the precise results.

\subsection{Results}

As above, we take $\hamvf_p$ to be the Hamilton vector field
of
$p=\abs{\xi}_g^2-1$ and $\e^{t \hamvf_p}$ its bicharacteristic flow
inside $p^{-1}(1)=S^*X$ (i.e., geodesic flow).  We continue to take
\begin{equation}
\mathcal{N}=\{\rho\in S^*X: \forall t\in\mathbb R, a\circ \e^{t\sH_p}(\rho)=0 \},
\end{equation}
and will add the further assumption that
$$
\overline{\pi(\cN)}\cap \supp a=\emptyset,
$$ 
where $\pi$ is projection $T^*X \to X.$
Thus there exists a non-empty open set $O_1$ such
that $\supp(a)\subset X\setminus O_1$, and  $\pi(\cN) \Subset
O_1$.   
The following Theorem is our main ``black box'' spectral estimate.  
\begin{thm}
\label{T:bbox1}
Assume  that for some $\delta \in (0,1)$
fixed and $K \in \ZZ$, there is a function $1 \leq \alpha(h) = \cO( h^{-K})$ such that
\[
\| (h^2\Delta_g +\i a-z)^{-1} \|_{L^2 \to L^2} \leq \frac{\alpha(h)}{h},
\]
for $z \in [1-\delta, 1+
\delta]$.  Then there exists $C, c_0>0$ such that 
\[
\| (h^2\Delta_g+\i ha -z )^{-1} \|_{L^2 \to L^2} \leq C \frac{\alpha(h)}{h},
\]
for $z \in [1-\delta,1+\delta] + \i [-c_0 , c_0 ] h/\alpha(h)$.
\end{thm}

When $\cN=\emptyset$, one has $\alpha(h)=\cO(1)$, while for  $\cN\neq
\emptyset $, one has $\alpha(h)\to\infty$ as $h\to 0$.  As a general heuristic, the ``larger'' the trapped set is, the
larger is $\alpha(h)$ when $h\to 0$, and the weaker the above global
estimate is---see section \ref{s:examples} below for examples.

\begin{rem} 
As discussed above, instead of the assumption on the model operator $h^2\Delta_g+\i W_1$ with complex 
absorption, we could just as well, by the results of
\cite{Datchev-Vasy:Gluing}, have made an assumption on a model operator 
in which the set $O_1$ is ``glued'' to non-compact ends of various forms. 
In particular, it would suffice to know the cut-off resolvent estimate on the real 
axis for the limiting resolvent
$$ 
\norm{\chi (h^2\Delta'-z+i0)^{-1}\chi}_{\cL(L^2,L^2)}\leq \frac{\alpha(h)}{h} 
$$ 
for a localizer $\chi$ equal to $1$ on $O_1$ and for $\Delta'$ 
the Laplacian on a manifold with Euclidean ends whose trapped set is contained 
in a set $O_1'$ isometric to $O_1.$ 
\end{rem} 
 
\begin{rem}
  The hypotheses of 
Theorem \ref{T:bbox1} 
  can be weakened to \emph{phase
    space} hypotheses, with $a$ a pseudodifferential operator (as in
  \cite{Sjo00}). 
  We have chosen to keep the damping as
  a function on the base in accordance with tradition and for the sake
  of brevity.

\end{rem}

\begin{rem}
The assumption that $\alpha(h) = \cO(h^{-K})$ is of a technical
nature.  It does not appear to be too restrictive, however, since
every {\it known} estimate for weakly unstable trapping satisfies this
assumption (see Section \ref{s:examples} below).  
Indeed, if the undamped set $\cN$ is at least weakly
semi-unstable, the results of \cite{Chr-surf-res} suggest that in fact
$\alpha(h)$ is always $\cO_\epsilon (h^{-1-\epsilon})$ for any $\epsilon>0$.

\end{rem}

\begin{rem}\label{R:non polynomial}

If $\alpha(h)$ is  not of polynomial nature, the proof of Theorem  \ref{T:bbox1} has to be slightly modified (see below). As a result, the final estimate we can obtain is weaker : there exists $C, c_0>0$ such that 
\[
\| (h^2\Delta_g+\i ha -z )^{-1} \|_{L^2 \to L^2} \leq C \frac{\alpha^2(h)}{h},
\]
for $z \in [1-\delta,1+\delta] + \i [-c_0 , c_0 ] h/\alpha^2(h)$. 
\end{rem}

In section~\ref{s:examples} we describe three different settings in
which our gluing results apply, in which the dynamics in a
neighborhood of the trapped
set are respectively
\begin{enumerate}
\item Normally hyperbolic
\item Degenerate hyperbolic
\item Hyperbolic with a condition on topological pressure.
\end{enumerate}
In addition to proving resolvent estimates in these settings, we discuss
applications to decay rates for solutions to the damped wave equation.

\section{Operators with complex absorbing potentials}

In this section, we collect some standard results about operators with
complex absorbing potentials.  Such a potential has a much stronger 
effect than damping, namely (in the microlocal absence of forcing),
that of annihilating semiclassical wavefront set completely along 
bicharacteristics passing through it in the forward direction.

We collect basic results about the resolvent of an operator with
complex absorbing potential.  This includes both existence of the
family and the basic propagation estimates, which tell us that the
complex absorbing potential kills off wavefront set under forward
propagation.  We begin with the definition of the
``resolvent:''
\begin{lemma}\label{lemma:resolvent}
Let $W\in C^\infty (X)$, $W\geq 0$ and $W$ not identically zero. Suppose 
that $$P_1=h^2\Delta_g+\i W$$ on $X$. Then $(P_1-z)^{-1}$ is a meromorphic family of bounded operators on $L^2$ for
all $z \in \CC,$ analytic in the closed lower half-plane.
\end{lemma}
\begin{proof}
We simply remark that $h^2
\Delta+1$ is invertible, and apply the analytic Fredholm theorem
to conclude that there exists a meromorphic resolvent family.  By the
Fredholm alternative, any pole has to correspond to nullspace of
$P_1-z.$  Since
$$
\Im \ang{P_1 u,u} = \ang{W u,u}-(\Im z)\norm{u}^2,
$$
for $u$ to be in the nullspace would imply that $\Im z\geq 0;$
equality would further require that $\ang{W u,u}=0$ which is
forbidden by unique continuation.
\end{proof}

Finally, we recall the microlocal bound of propagation through
trapping by Datchev-Vasy \cite{Datchev-Vasy:Trapping}, as well as basic backward propagation of
singularities in the presence of complex absorption (see also Lemma
A.2 of \cite{NoZw09_1} for the latter).  In this context we will say that a
bicharacteristic $\gamma$ (by bicharacteristic we always mean an
integral curve of $H_{\re p_1}$ in
the characteristic set of $\re p_1-\re z$ where $p_1$ is the
semiclassical principal symbol of $P_1$), or a point on $\gamma$,
is \emph{ non-trapped} if $W(\gamma(T))>0$ for some $T\in\RR$, and is
\emph{trapped} otherwise.
We
 say it is
 \emph{forward non-trapped} if $W(\gamma(T))>0$ for some $T>0.$
In the 
terminology of \cite{Datchev-Vasy:Trapping} we say that\footnote{We have opposite signs for imaginary parts of $p_1$
  relative to \cite{Datchev-Vasy:Trapping}, so incoming and outgoing are
  interchanged.} the resolvent $(P_1-z)^{-1}$ is {\em semiclassically incoming} 
{\em with a loss of $h^{-1}$} provided that whenever $q\in T^*X$ is 
on a forward non-trapped bicharacteristic $\gamma$ of $P_1$ and $f=\cO(1)$ on 
$\gamma|_{[0,T]}$ then $(P_1-z)^{-1} f$ is $\cO(h^{-1})$ at $q$. 

\begin{lemma}(See \cite{Datchev-Vasy:Trapping}.) \label{lemma:Datchev-Vasy:trapping} 
Suppose 
that $P_1=h^2\Delta_g+\i W$ on $X$, $W\geq 0$,\ and $z\in \CC$ such
that $\im  z=\cO(h^\infty)$, $\re z\in[1-\delta,1+\delta]$ for
$0<\delta<1$ fixed. Assume that the resolvent is polynomially 
bounded, i.e., 
$$\norm{(P_1-z)^{-1}}_{\cL(L^2,L^2)} \leq C h^{-K}\quad 
\text{ for some } K.
$$ 
Then $(P_1-z)^{-1}$ is semiclassically incoming
with a loss of $h^{-1}$.
In particular, 
if $W(\gamma(T))>0$ for some $T>0$ and $\WFh(f)$ is disjoint from $\gamma|_{[0,T]}$, then $(P_1-z)^{-1} f=\cO(h^\infty)$ at $q$. 
 
Further, if $\chi\in\CI(X)$ and $T^*_{\supp\chi} X$ contains no
trapped points, then $\chi (P_1-z)^{-1}\chi$ is $\cO(h^{-1})$. 
\end{lemma} 

A further result that is of crucial importance in avoiding losses
in 
our estimates is the following result of \cite{DaVa-trapping-2}:

\begin{lemma}\label{lemma:Datchev-Vasy:trapping2}(See \cite{DaVa-trapping-2}.)
With the notation of Lemma~\ref{lemma:Datchev-Vasy:trapping}, if
$$
\norm{(P_1-z)^{-1}}_{\cL(L^2,L^2)} \leq \frac{\alpha(h)}{h}
$$ 
and
if $\chi\in\CI(X)$ and $T^*_{\supp\chi} X$ contains no
trapped points, then for some $C>0$,
$$
\norm{(P_1-z)^{-1}\chi}_{\cL(L^2,L^2)} \leq C\frac{\sqrt{\alpha(h)}}{h}.
$$
\end{lemma}

\section{Propagation and damping estimates}

We now switch from complex absorbing potentials back to damping:
set
$$
P=P(h)=h^2\Delta_g +\i h a
$$
and consider the equation 
$$
(P-z)u=f\,,\quad z\in[1-\delta,1+\delta]\,.
$$
We also set
$\Sigma=p^{-1}(0)\subset T^*X$ where $p=\sigma_h(P)$. 
Let us start by recalling a classical result about propagation
estimates (see, e.g.\ Theorem~12.5 of \cite{Zworski:semiclassical} for
a proof by conjugation to normal form; an alternative is the usual
commutator argument as described in \cite{Hormander:enseignement} in
the homogeneous setting and \cite{Chr-disp-1} in the semiclassical setting):
\begin{lemma}\label{l:positive commut}
Suppose $q\in\Sigma$ and for some $T>0$ the forward 
  bicharacteristic $\exp([0,T]\sH_p)(q)$ is disjoint from a compact
  set $K$. Then there are $Q,Q'$ which are elliptic at $q$, resp.\
  $\exp(T\sH_p)(q)$, with $\WFh'(Q')\cap K=\emptyset$ such that
  $\|Qu\|\leq \|Q'u\|+Ch^{-1}\|f\|$, $u=(P-z)^{-1}f$.
\end{lemma}

\section{Proof of Theorem \ref{T:bbox1}}\label{s:bbox1}

For the moment, we consider only the case where $z\in
[1-\delta,1+\delta]$. We have
$$
P=h^2\Lap_g+\i ha
$$
and we additionally write $$P_1 = h^2\Delta_g + \i a $$ for the operator with
damping replaced by absorption.
Choose an open set $V_1$ such that $\cN \Subset V_1 \Subset O_1$.  
Let $B_1, \varphi \in C_0^\infty (X)$ be  smooth functions with $B_1|_{V_1} =
1$, $\supp B_1 \subset O_1$,
$\varphi = 1$ on $\supp \nabla B_1$ and $\supp\varphi \cap
\cN=\emptyset$.   
We observe that $\cN$ satisfies the assumptions of Lemma~\ref{lemma:Datchev-Vasy:trapping2}, so that
\[
\| (P_1 -z)^{-1} \varphi u \| \leq C h^{-1} \sqrt{\alpha(h)}
\| \varphi u \|.
\]
Then, noticing that $a$ and
$B_1$ have disjoint supports, we have
\begin{equation}\label{B1estimate}
\begin{aligned}
\| B_1 u \| & = \| (P_1-z)^{-1} (P_1-z) B_1 u \| \\
& = \| (P_1-z)^{-1} (P-z) B_1 u \| \\
& = \| (P_1-z)^{-1} (B_1 (P -z) + [P, B_1] ) u \| \\
& \leq \| (P_1-z)^{-1} B_1 (P -z) u \| + \| (P_1 -z)^{-1} \varphi [P,
B_1]  \varphi u \| \\
& \leq \frac{\alpha(h)}{h} \| B_1 (P-z) u \| + C \sqrt{ \alpha(h) } \|
\varphi u \| 
\end{aligned}
\end{equation}

Since for $\rho \notin \cN$ the curve $e^{t\hamvf_p}(\rho)$ passes
through $\{a>0\},$ each such bicharacteristic curve must certainly enter the compact
set $X \backslash O_1.$  Thus by compactness, there exists $\ep_0>0$
such that every such curve passes through $\{a\geq \ep_0\}.$  We now take
a cutoff function $\chi\geq 0$ with $\supp \chi \subset \{a>\ep_0/2\},$
and $\chi=1$ whenever $a\geq \ep_0;$ hence every controlled geodesic
passes through $\{\chi=1\}.$

We next recall a classical lemma concerning the  propagation of
singularities in the presence of geometric control. (See
\cite{BuZw04}, which builds on a semiclassical version of \cite{Hormander:enseignement},
proved in \cite{SjZw02}.)  This is a slight variation on Lemma
\ref{l:positive commut}, and can of course also be proved using the original
positive commutator argument (see \cite{Chr-disp-1}).
\begin{lemma} (See \cite{BuZw04})
Let $U$ be an open neighbourhood of $\cN$, $\chi\in\CI(X)$ as above.
If  $B_0=\Psi_h^{0,0}(X)$ is such
that $\WF'_h(B_0)\subset T^*X\setminus U$, then for $z$ real near 1, 
$$
\|B_0u\|\leq \frac{C}{h}\|(P -z)u\| +\|\chi u\|+\cO(h^\infty )\|u\|
$$
\end{lemma}
Since $\supp \varphi$ and $\supp (1-B_1)$ lie inside the controlled region, we
can write : 
\[
\| (I - B_1) u \| \leq C h^{-1} \| (P-z) u \| + C\|\chi u \| +\cO(h^\infty )\|u\|
\] 
and
\[
\| \varphi u \| \leq C h^{-1} \|(P-z) u \| + C\| \chi u \| +\cO(h^\infty )\|u\|\,.
\]
But
\begin{align*}
\| \chi u \|^2 & \leq C \langle a u , u \rangle\\
& = Ch^{-1} \Im \langle (P -z) u, u \rangle \\
& \leq Ch^{-1}\| (P -z ) u \| \| u \|.
\end{align*}
Starting with \eqref{B1estimate}, we deduce from the above inequalities that
\begin{align*}
\|B_1 u\|^2&\leq C\Big( \frac{\alpha^2(h)}{h^2}\|(P-z)u\|^2 +C
\alpha(h) \|\varphi u\|^2\Big)\\
& \leq C \Big( \frac{\alpha^2(h)}{h^2}\|(P-z)u\|^2 +
C\alpha(h)\|\chi u\|^2+\cO(h^\infty )\|u\|^2\Big).
\end{align*}
Hence, we have 
\begin{align*}
\| u \|^2 & \leq C ( \| B_1 u \|^2 + \| (I - B_1) u \|^2 ) \\
& \leq C \Big( \frac{\alpha^2(h)}{h^2} \| (P-z) u \|^2 +  \alpha(h) \| \chi
u \|^2 +\cO(h^\infty) \| u \|^2 \Big) \\
& \leq C \Big(  \frac{\alpha^2(h)}{h^2} \|(P-z) u \|^2 + \frac{\alpha(h)}{h} \| (P-z) u \| \| u \|\\
& \quad +\cO(h^\infty)\|u\|^2 \Big) \\
& \leq C \Big(  \frac{\alpha^2(h)}{h^2} \| (P-z) u \|^2
+ \frac{4\epsilon^{-1} \alpha^2(h)}{h^2} \| (P-z) u \|^2  + \epsilon\| u \|^2 \\
& \quad +\cO(h^\infty)\|u\|^2 \Big).
\end{align*}
If $\epsilon$ is small, we can absorb the last two terms in the
above inequality on the left hand side, and we obtain
\[
\| u \| \leq C \frac{\alpha(h)}{h} \| (P-z) u \|.
\] 
Now simply observe that by the triangle inequality this bound is still valid if we add to $z$ an
imaginary part that satisfies
$$
|\Im z |\leq h \alpha^{-1}(h) C' 
$$
for $C'$ such that $C'C<1$, and this concludes the proof of the
theorem. 

\begin{rem}

If $\alpha(h)$ is not of polynomial nature, then Lemma~\ref{lemma:Datchev-Vasy:trapping2} cannot be used. As a result, the square root in Equation \eqref{B1estimate} must be removed. The rest of the argument is the same, and we end up with the estimate given in Remark \ref{R:non polynomial}. Note also that the energy decay rates for the damped wave equation are of course  weaker than in the case where Lemma~\ref{lemma:Datchev-Vasy:trapping2} can be applied.
\end{rem}

In order to apply Theorem \ref{T:bbox1} to the situation of the stationary
damped wave operator, we now state a corollary where the Schr\"odinger
operator depends mildly on $z$.

\begin{cor}
\label{C:bb-1}
Let $X$ be a compact manifold without boundary, and let $\tP(h,z)$ be the
modified operator 
$$\tP(h,z)=h^2\Delta_g +\i h \sqrt{z}a -z\,.$$

Assume  that for some $\delta \in (0,1)$
fixed  there is a function $1 \leq \alpha(h) = \cO(h^{-K} )$, for some
$K \in \ZZ$, such that
\[
\| (h^2\Delta_g +\i a-z)^{-1} \|_{L^2 \to L^2} \leq \frac{\alpha(h)}{h},
\]
for $z \in [1-\delta, 1+
\delta]$.  Then there exists $C, c_0  >0$ such that 
\[
\| \tP(h,z)^{-1} \|_{L^2 \to L^2} \leq C \frac{\alpha(h)}{h},
\]
for 
\[
z \in  [1-c_0/\alpha(h), 1 + c_0/\alpha(h)] + \i[-c_0,c_0]h/\alpha(h).
\]

\end{cor}

\begin{proof} The real part of the perturbation is manifestly bounded by a
  small multiple of $h/\alpha(h)$ and can thus be perturbed away by
  Neumann series.  It thus only remains
to check that the size of the imaginary part of the
perturbation: 
\[
\Re (h a - \sqrt{z} ha) - \Im z
\]
can also be made much smaller than $h / \alpha(h)$.

Take $\sqrt{z} = 1 + r + \i \beta$ for $r, \beta$ to be determined.  Then $z =
(1+r)^2 - \beta^2 + 2 (1 + r) \i \beta$.  Then
\[
\Re (h a - \sqrt{z} ha) - \Im z = h a (1-(1 + r)) - 2(1 + r) \beta =
\epsilon h / \alpha(h)
\]
for $\epsilon>0$ small if, say, $| \beta | \leq \epsilon h/ 2
\alpha(h)$ and $| r | \leq \alpha^{-1}(h)$.  
Squaring, we obtain
\[
z \in [1-c_0/\alpha(h), 1 + c_0/\alpha(h)] + \i[-c_0,c_0]h/\alpha(h).
\]

\end{proof}

\section{Examples}\label{s:examples}
In this section, we briefly outline some known microlocal resolvent
estimates, state the two different stationary damped wave operator
estimates, and then draw conclusions about solutions to the damped
wave equation \eqref{eq:dwe}.

\subsection{A normally hyperbolic trapped set}

In this section, we treat the case in which the trapped set is a
smooth manifold in $S^*X$ around which the dynamics is \emph{normally
  hyperbolic}.  In this case, estimates for the resolvent with a
complex absorbing potential have been obtained by Wunsch-Zworski
\cite{WuZw}.  A particular case of interest is the ``photon sphere''
for the Kerr black hole geometry, where the phase space is $6$-dimensional $\cN$ is a symplectic
submanifold, diffeomorphic to $T^* S^2$---see section 2 of \cite{WuZw}
for details on this application, and
\cite{Vasy-Dyatlov:Microlocal-Kerr}
for placing it in actual Kerr-de Sitter space.  Another special case is of course that of
a single hyperbolic closed geodesic (discussed further in the following section).

The precise formulation of normal hyperbolicity used here is as
follows: we define the backward-forward trapped sets by
$$
\Gamma_{\pm}=\{\rho\in T^*X: \forall t\gtrless 0, a\circ \e^{t\sH_p}(\rho)=0 \}
$$
Then of course
$$
\cN=\Gamma_+ \cap \Gamma_-,
$$
where we have now ceased to restrict to a single energy surface (so
$\cN \subset T^*X$ is a homogeneous subset in view of the
homogeneity
of $p$) in order to employ the
terminology of symplectic geometry more easily.

We make the following assumptions on this intersection:
\begin{enumerate}
\item \label{gammas}$\Gamma_\pm$ are codimension-one 
smooth
manifolds intersecting transversely at $\cN.$ (It
is not difficult to verify that $\Gamma_\pm$ must then be coisotropic and
$\cN$ symplectic.)

\item \label{normflow} The flow is hyperbolic in the normal directions
to $K$ within the energy surface $S^*X$: there exist subbundles $E^\pm$ of
$T_{\cN} (\Gamma_\pm)$ such that at $p \in S^*X$
$$
T_{\cN} \Gamma_\pm=T \cN \oplus E^\pm,
$$
where $$d(\exp(\hamvf_p) :E^\pm \to E^\pm$$ and
there exists $\theta>0$ such that
\begin{equation}
\label{eq:normh}
\| {d(\exp(\hamvf_p )(v)} \|  \leq C e^{-\theta | t | } \| {v} \| \
\text{for all } v \in E^{\mp},\ \pm t \geq 0.
\end{equation}
\end{enumerate}
As discussed in \cite{WuZw}, these hypotheses as stated are not
structurally stable, but they do follow (at least up to loss of
derivatives) from the more stringent
hypothesis that the dynamics be \emph{$r$-normally-hyperbolic} for
every $r$ in the sense of
\cite[Definition 4]{Hirsch-Pugh-Shub}.  This implication, and the
structural stability of the hypotheses, follows from a deep
theorem of Hirsch-Pugh-Shub \cite{Hirsch-Pugh-Shub} and Fenichel \cite{Fenichel}.

As a consequence of the estimates of \cite{WuZw} for resolvents, we
then obtain the following estimate for the damped operator:

\begin{thm}\label{T:normhyp}
Let $(X,g)$ satisfy the dynamical conditions enumerated above.  Then
we have
\[
\| (h^2\Delta_g+\i h 
a -z )^{-1} \|_{L^2 \to L^2} \leq C \frac{\abs{\log h}}{h},
\]
for $z \in [1-\delta,1+\delta] + \i [-c_0 , c_0 ] h/\abs{\log h}$.
\end{thm}

This estimate, or more precisely its refinement in
Corollary~\ref{C:bb-1},  
 provide a corresponding energy decay
estimate for solutions to the damped wave equation \eqref{eq:dwe}.  In
order to avoid irritating issues of projecting away from constant
subspaces, etc., we assume that $u_0 \equiv u(x,0)= 0$.

\begin{cor}
\label{C:damp1}
Assume the hypotheses of Theorem \ref{T:normhyp} hold, and let $u$ be
a solution to \eqref{eq:dwe} with $u_0 = 0$, and $u_1 \in H^s$ for
some $s\in (0,2]$.   Then there exists a constant $C = C_{
  s}>0$ such that
\[
E(u,t) \leq C e^{ -s t^{1/2}/C} \|u_1\|_{H^s}^2.
\]

\end{cor}

A simple situation in which the hypotheses of
Theorem~\ref{T:bbox1} are satisfied is that of a connected compact
manifold of the form $X=X_0\cup X_1$ with $X_1$ open and $X_0$ isometric
to a warped product $\RR_u\times S^{n-1}_\theta$ with metric
$$
g=du^2+ \cosh^2 u\, d\theta^2.
$$
We take $a\in \CI(X)$ to be identically $1$ on $X_1$ as well as equal
to $1$ for $\abs{x}>1$ in $X_0.$  This class of manifolds thus
includes the ``peanut of rotation'' shown in
Figure~\ref{figure:peanut} as well as its
higher dimensional generalizations.
\begin{figure}
\hfill
\centerline{\input{revised.peanut.tex}}
\caption{\label{figure:peanut} The ``peanut of rotation''. }
\hfill
\end{figure}
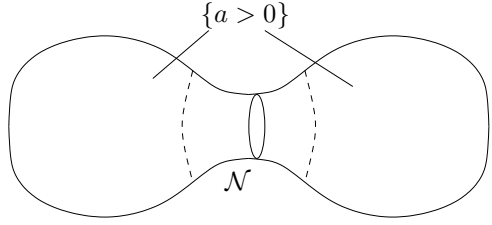

Then the trapped set is easily seen to be $\cN=\{u=0, \xi=0\}$ where
$\xi$ is the cotangant variable dual to $u,$ and the function $x=2-u^2$
satisfies the convexity hypotheses.  The flow on $\cN$ is normally
hyperbolic, with the stable and unstable manifolds being the two
components of the set
$$
\big(\xi^2+\frac{\abs{\eta}_{S^{n-1}}^2}{\cosh^2 u}\big)=\abs{\eta}_{S^{n-1}}^2
$$
i.e., by the intersection of the condition that energy and angular
momentum match their values on $\cN.$
This is an example of a
normally hyperbolic trapped set, and hence both parts of
Theorem~\ref{T:normhyp} apply.

\subsection{A trapped set with degenerate hyperbolicity}

In this section, we study a variant of the normally hyperbolic case,
in which the intersection of stable and unstable manifolds is no
longer transverse, hence the results of \cite{WuZw} no longer apply.
This is the case of a surface of rotation with a \emph{degenerate
  hyperbolic} closed orbit.

Our example manifold is a topological torus
$X=[-1,1]_x\times \SS^1_\theta$, 
equipped with the metric 
\begin{equation}\label{E:metrictorus} 
ds^2 = d x^2 + A^2(x) d \theta^2
\end{equation}
where near $x=0$,  
$$A(x) = (1+|x|^{2m} )^{\frac{1}{2m}}$$
and $m$ is an integer $\geq 1$.   This manifold has a ``fatter'' part and a
``thinner'' part.  At the thickest, there is an elliptic geodesic, and
at the thinnest part, where $x=0$, there is an unstable geodesic, which we denote
by $\gamma$.
If $m\geq  2$,  the Gaussian curvature is chosen to vanish to a finite order at the unstable
geodesic, hence the geodesic is {\it degenerately hyperbolic}.  If
the Gaussian curvature is strictly negative ($m=1$) in a neighbourhood of the
thinnest part, the geodesic is non-degenerate; the geometry of a
single closed hyperbolic geodesic has been extensively studied
in \cite{CdVP-I, CdVP-II, BuZw04, Chr-NC, Chr-NC-erratum, Chr-QMNC, BuCh-dw} and
others.  As is seen in the previous section, in this non-degenerate hyperbolic case the energy
decays sub-exponentially with derivative loss; in \cite{BuCh-dw},
it is shown that the sub-exponential decay rate is sharp. Based on the sharp polynomial loss in local smoothing and resolvent
estimates in
\cite{ChWu-lsm}, Theorem \ref{T:bbox1}
shows that for the degenerate hyperbolic periodic
geodesic, we have the following estimates.
\begin{thm}\label{thm:resol_degen}
Let $X$ be as above, and suppose $a(x)$ controls $X$ geometrically outside a sufficiently small
neighbourhood $U \supset \gamma$, so that  $\cN=\{\gamma\}$. Then  
\[
\| (h^2\Delta_g + \i h 
a -z)^{-1} u \| \leq C h^{-2m/(m+1)} \| u \|
\]
for $z \in [1-\delta, 1 + \delta] + i[-c_0, c_0] h^{2m/(m+1)}$.
\end{thm}

As in the previous subsection, we deduce from this resolvent
estimate an energy decay estimate for solutions to the damped wave
equation.

\begin{cor}
\label{C:damp2}
Assume the hypotheses of Theorem \ref{thm:resol_degen} hold, and let $u$ be
a solution to \eqref{eq:dwe} with $u_0 = 0$, and $u_1 \in H^s$ for
some $s \in(0,2]$.  Then there exists a constant $C = C_{s}
>0$ such that
\[
E(u,t) \leq C \left( \frac{t^{\frac{m+1}{m-1}
    }}{(\log(2+t))^{\frac{3(m+1)^2}{2(m-1)^2} } } \right)^{-s} \|u_1\|_{H^s}^2.
\]

\end{cor}

\subsection{Hyperbolic trapped set with small topological pressure}

In this section, we assume that the trapped set $\cN$  has a hyperbolic structure, and that the topological pressure of half
the unstable Jacobian on the trapped set is negative. Roughly, this means that the set $\cN$ is rather thin, 
or filamentary: in dimension 2,
this is for instance equivalent to require that $\cN$ has Hausdorff dimension
$<2$. The simplest case to have in mind is a single, closed
hyperbolic orbit. We then build on \cite{NoZw09_1} to get
resolvent estimates near the trapped set, which we extend to the
global manifold with our different methods. 

We briefly recall here the above dynamical notions. By definition, the hyperbolicity of $\cN\subset S^*X$ means that  for any $\rho\in\cN$, the tangent space $T_{\rho}\cN$ splits
into \emph{flow}, \emph{stable} and \emph{unstable} subspaces\[
T_{\rho}\cN=\RR \sH_{p}\oplus E^{s}_{\rho}\oplus E^{u}_{\rho}\,.\]
If $X$ is of dimension $d$, the spaces $E^{s}_{\rho}$ and $E^{u}_{\rho}$ are $d-1$ dimensional,
and are preserved under the flow map: \[
\forall t\in\RR,\ \ d\e^{t\sH_p} (E^{s}_{\rho})=E^{s}_{\e^{t\sH_p}(\rho)},\quad d \e^{t\sH_p}(E^{u}_\rho)=E^{u}_{\e^{t\sH_p}(\rho)}.\]
Moreover, there exist $C,\lambda>0$ such that \begin{eqnarray}
i) &  & \|d\e^{t\sH_p}(v)\|\leq C\e^{-\lambda t}\|v\|,\ \textrm{\ for\ all\ }v\in E^{s}_{\rho},\ t\geq0\nonumber \\
ii) &  & \|d\e^{-t\sH_p}(v)\|\leq C\e^{-\lambda t}\|v\|,\ \textrm{\ for\ all\ }v\in E^{u}_{\rho},\ t\geq0.\label{eq:Unstable}\end{eqnarray}
One can show that there exist a metric on $T^{*}X$ call the \emph{adapted
metric}, for which one can take $C=1$ in the preceding equations.

The above properties allow us to define the unstable
Jacobian. The adapted metric on $T^{*}X$ induces a volume form
$\Omega_{\rho}$ on any $d$ dimensional subspace of $T(T_{\rho}^{*}X)$.
Using $\Omega_{\rho}$, we can define the unstable Jacobian at $\rho$
for time $t$. Let us define the weak-stable and weak-unstable subspaces
at $\rho$ by 
\[
E^{s,0}_{\rho}=E^{s}_{\rho}\oplus\RR \sH_{p}\,,\quad E^{u,0}_\rho=E^{u}_{\rho}\oplus\RR \sH_{p}.\]
We set \[
J_{t}^{u}(\rho)=\det d\e^{-t\sH_p}|_{E^{u,0}_{\e^{t\sH_p}(\rho)}}=\frac{\Omega_{\rho}(d\e^{-t\sH_p}v_{1}\wedge\dots\wedge d\e^{-t\sH_p}v_{d})}{\Omega_{\e^{t\sH_p}(\rho)}(v_{1}\wedge\dots\wedge v_{d})}\,,\ \ \ \ J^{u}(\rho):= J_{1}^{u}(\rho),\]
where $(v_{1},\dots,v_{d})$ can be any basis of $E^{u,0}_{\e^{t\sH_p}(\rho)}$.
While we do not necessarily have $J^{u}(\rho)<1$, it is true that
$J_{t}^{u}(\rho)$ decays exponentially as $t\to+\infty$. 

We denote by $\Pr_\cN$ the topological pressure functional on the
closed, invariant set $\cN$. We briefly recall a definition, see
\cite{Wal75}, \cite{NoZw09_1} for more material. If $f$ is a
continuous function on $\cN$, $n$ an integer and $\epsilon>0$,
define
$$
Z_{n,\epsilon}(f)=\sup_{\cS} \left\{ \sum_{\rho\in\cS} \exp {\sum_{k=0}^{n-1} f\circ \e^{k\sH_p}(\rho)} \right\}
$$
where the supremum is taken over all the  $(\epsilon,n)$ separated subsets $\cS$. The topological pressure of $f$ on $\cN$ is then 
$$
\Pr_\cN(f):=\lim_{\epsilon\to0}\limsup_{n\to\infty} \frac{1}{n} \log Z_{n,\epsilon}(f)\,.
$$
Our main assumption here is that the topological pressure of $\frac{1}{2}\log J^u$
on $\cN$ is negative, namely:
$$
\Pr_\cN(\frac{1}{2}\log J^u )<0\,.
$$ 
For some $\delta >0$ small enough, this imply the following resolvent estimate: 
\begin{equation}
\forall z\in [1-\delta, 1+\delta],\quad \|(h^2\Delta_g
+\i a-z)^{-1}\|\leq C \frac{\abs{\log h}}{h}
\end{equation}

This estimate is already contained in
\cite{NoZw09_1}, modulo two minor simplifications in our case: the manifold
is compact, and infinity is replaced with the absorbing potential
$\i a$, which control everything outside the trapped set --
\cite{Datchev-Vasy:Gluing} shows explicitly that the estimate, with the more complicated
geometry at infinity, of \cite{NoZw09_1} implies the slightly simpler
complex absorption result.
Using Theorem
\ref{T:bbox1}, 
we immediately deduce the following result: 
\begin{thm}\label{T:pressure}
Let $X$ be a compact manifold, and suppose that  $a\geq 0$ controls $X$ except on $\cN$,
which is assumed to be hyperbolic with the property that 
$$
\Pr _\cN(\frac{1}{2}\log J^u )<0\,.
$$
For $\delta>0$ small enough, there is $h_0$ and $c_0>0$ such that
 for $h\leq h_0$ and $z\in [1-\delta, 1+\delta]+\i [-c_0, c_0]
 \frac{h}{|\log h|}$ we have 
$$
\|(h^2 \Delta_g +\i h 
a-z)^{-1} \|\leq C \frac{\abs{\log h}}{h}\,.
$$
\end{thm}
In particular, there is no spectrum near the real axis in a region of
size $h/\abs{\log h}$. As the resolvent estimate is the same order as that in Theorem
\ref{T:normhyp}, we deduce the same energy decay estimates as in
Corollary \ref{C:damp1}.

\section{From resolvent estimates to the damped wave equation and energy decay}\label{s:dwe}

In this section, we show how to move from a high energy resolvent
estimate to an energy decay estimate for the damped wave equation,
proving Corollaries \ref{C:damp1} and \ref{C:damp2}.
To estimate the energy decay for the damped wave equation, as usual we
rewrite it as a first-order evolution problem : if  $\bs u=(u,\d_{t}u)$
one can write  \eqref{eq:dwe} as 
\begin{equation}\label{eq:DWE2x2}
\d_{t}\bs u=\i\cB\bs u,\quad\cB=\left(\begin{array}{cc}
0 & -\i \Id\\
\i \Delta_{g} & \i  a\end{array}\right)\,.
\end{equation}
The evolution group $\e^{\i t\cB}$ maps initial data $(u_0,u_1)\in
H:= H^1(X)\times H^0(X)$  to a solution $(u, \d_t u)$ of
\eqref{eq:DWE2x2} where $u$ solves \eqref{eq:dwe}. For $s>0$, define 
$$
\|u\|_s:=\|u_0\|_{H^{1+s}(X)}+\|u_1\|_{H^s(X)}
$$
It is not hard to see that  if we can prove 
\begin{equation}\label{eq:DecayEstimate}
\left\|  \e^{\i t\cB}(1-\i \cB)^{-s}   \right\|_{L^2(X)\to
  L^2(X)}^2\leq f(t)
\end{equation}
then we can deduce a decay rate for the energy : 
$$
E(u,t)\leq f(t) \|u\|_s^2
$$
It turns out that we can obtain bounds such as
\eqref{eq:DecayEstimate} if we have estimates on the high-frequency
resolvent $(\lambda-\cB)^{-1}$, $|\lambda| \to\infty$.  

To see this, we recall the following setup from \cite{Chr-wave-2}.  Now suppose 
$ (\lambda - \cB)^{-1}$ continues holomorphically to a
neighbourhood of the
region
\be
\Omega = \left\{ \lambda \in \mathbb{C} : | \Im \lambda | \leq
  \left\{ \begin{array}{ll} C_1, & | \Re \lambda | \leq C_2 \\
      P(| \Re \lambda | ), & | \Re \lambda | \geq C_2, \end{array}
  \right. \right \},
\ee
where $P( | \Re \lambda |) >0$ and is monotone decreasing (or constant) as $| \Re \lambda |
\to \infty$, $P(C_2) = C_1$, and assume for simplicity that $\partial \Omega$ is smooth.  Assume
\ben
\label{res-est-101}
\|  (\lambda - \cB)^{-1}  \|_{H \to H} \leq G( | \Re \lambda |)
\een
for $\lambda \in \Omega$, where $G(| \Re \lambda| ) = \mathcal{O}( | \Re \lambda |^N)$ for
some $N \geq 0$.  

\begin{thm}
\label{T:res-to-decay}
Suppose $\cB$ satisfies all the assumptions above, and let $k \in \NN$,
$k > N+1$.  Then for any $F(t)>0$, monotone increasing, 
satisfying
\ben
\label{F-cond}
F(t)^{(k+1)/2} \leq \exp (t P(F(t))),
\een
there is a constant $C>0$ such that
\ben
\label{abs-decay-est}
\left\| \frac{ \e^{\i t \cB}}{(1 -\i\cB)^k} 
\right\|_{H \to H} \leq C F(t)^{-k/2}.
\een
\end{thm}

In all cases considered in this paper, we have semiclassical resolvent
estimates
\[
\| (h^2 \Delta_g +\i \sqrt{z}h a -z)^{-1} \|_{L^2 \to L^2} \leq
\frac{\alpha(h)}{h},\quad z \sim 1 + \i h/\alpha(h),
\]
If we rescale
\[
\tau^2 = \frac{z}{h^2},
\]
then our resolvent estimates become
\[
\| (\Delta_g + \i \tau a - \tau^2 )^{-1} \|_{L^2 \to L^2} \leq
\frac{\alpha(|\tau|^{-1})}{|\tau|}.
\]
for $\Im \tau \sim (\alpha( |\Re \tau|^{-1} ))^{-1}$.  By interpolation,
this implies for $0 \leq j \leq 2$, 
\[
\| (\Delta_g + \i \tau a - \tau^2 )^{-1} \|_{H^{s} \to H^{s+j}} \leq
| \tau|^{j-1} {\alpha(|\tau|^{-1})}.
\]
Hence, with $\cB$ as above and $H = H^1 \times H^0$, a simple
calculation yields 
\[
\| (\lambda - \cB )^{-1} \|_{H \to H} \leq \alpha ( | \lambda |^{-1} ).
\]

For Corollary \ref{C:damp1}, we take 
$\alpha( |
\lambda|^{-1} ) = \log ( 2 + | \lambda |)$.  Then $k = 2$ suffices,
$P(r) =
\log^{-1}(r)$, and
hence we may take
\[
F(t)= 
\e^{t^{1/2}/C}.
\]
This recovers the endpoint estimate $s=2$.  To get the intermediate
estimates for $s \in (0,2)$ we interpolate with the trivial estimate
\[
E(u,t) \leq E(u, 0).
\]

For Corollary \ref{C:damp2}, 
we have $\alpha( | \lambda |^{-1}) = | \lambda
|^{(m-1)/(m+1)}$, $N = (m-1)/(m+1) <1$, so that $k = 2$, and $P(r) =
r^{(1-m)/(m+1)}$.  We try 
\[
F(t) = \frac{t^s}{\log^q(t)},
\]
and insist 
\[
t^{3s/2} \log^{-3q/2} (t) \leq \exp(t t^{s(1-m)/(m+1)}
\log^{q(m-1)/(m+1)}(t) ),
\]
which is satisfied if
$$
s = \frac{m+1}{(m-1)}
$$
 and 
$$
q = \frac{3(m+1)^2}{2(m-1)^2}.
$$ 
Again interpolating with the trivial estimate proves the Corollary.

\end{document}

%% file: revised.peanut.tex
\begin{picture}(0,0)%
\includegraphics{revised.peanut.pstex}%
\end{picture}%
\setlength{\unitlength}{2644sp}%
\begingroup\makeatletter\ifx\SetFigFont\undefined%
\gdef\SetFigFont#1#2#3#4#5{%
  \reset@font\fontsize{#1}{#2pt}%
  \fontfamily{#3}\fontseries{#4}\fontshape{#5}%
  \selectfont}%
\fi\endgroup%
\begin{picture}(4524,2032)(3739,-4223)
\put(5776,-3961){\makebox(0,0)[lb]{\smash{{\SetFigFont{10}{12.0}{\rmdefault}{\mddefault}{\updefault}{\color[rgb]{0,0,0}$\mathcal{N}$}%
}}}}
\put(5551,-2386){\makebox(0,0)[lb]{\smash{{\SetFigFont{10}{12.0}{\rmdefault}{\mddefault}{\updefault}{\color[rgb]{0,0,0}$\{a>0\}$}%
}}}}
\end{picture}%

%% file: damp.bbl
\begin{thebibliography}{XX}

\bibitem{Ana10}
N.~Anantharaman.
\newblock Spectral deviations for the damped wave equation.
\newblock {\em G.A.F.A.}, 20:593--626, 2010.

\bibitem{BaLeRa92}
C.~Bardos, G.~Lebeau, and J.~Rauch.
\newblock Sharp sufficient conditions for the observation, control and
  stabilization of waves from the boundary.
\newblock {\em SIAM J. Control and Optimization}, 30(5):1024--1065,
1992.

\bibitem{BuCh-dw} N. Burq and H. Christianson.  Imperfect geometric control and
  overdamping for the damped wave equation.  {\it In preparation.}

\bibitem{BuZw04} N. Burq and M. Zworski, Geometric control in the
  presence of a black box, {\em Journal of the A.M.S.} 17(2):443--471, 2004. 



\bibitem{CPV} F. Cardoso, G. Popov, and G. Vodev.  Semi-classical Resolvent Estimates for the
Schr\"odinger Operator on Non-compact Complete Riemannian
Manifolds. {\em Bull. Braz.
Math Soc.} 35, No. 3, 2004, p. 333--344.

\bibitem{Chr-NC}
H.~Christianson.
\newblock Semiclassical Non-concentration near hyperbolic orbits.
{\em J. Funct. Anal.}  246(2):145--195, 2007.

\bibitem{Chr-NC-erratum}   
H.~Christianson.
\newblock Corrigendum to semiclassical non-concentration near hyperbolic
  orbits.
\newblock {\em J. Funct. Anal.}, 246(2):145--195, 2007.

\bibitem{Chr-disp-1} H. Christianson.  Dispersive Estimates for
  Manifolds with one Trapped Orbit. {\em Comm. Partial Differential Equations} Vol. 33, (2008) no. 7, 1147--1174.



\bibitem{Chr-wave-2} H. Christianson, Applications of Cutoff Resolvent
  Estimates to the Wave Equation. {\em Math. Res. Lett. Vol.} 16 (2009), no. 4, 577--590.

\bibitem{Chr-QMNC}  H. Christianson.  Quantum Monodromy and
  Non-concentration Near a Closed Semi-hyperbolic Orbit. {\em
    Trans. Amer. Math. Soc.}  Vol. 363 (2011), no. 7, 3373--3438.

\bibitem{Chr-surf-res} H. Christianson.  On High-frequency Resolvent
  Estimates in Spherically Symmetric Media.  {\it In preparation.}


\bibitem{ChWu-lsm} H. Christianson and J. Wunsch, Local Smoothing for
  the Schrodinger Equation with a Prescribed Loss. {\em
    Amer. J. Math.}, to appear.



\bibitem{CdVP-I} Y. Colin de Verdi\'ere  and B. Parisse. Equilibre
  Instable en R\'egime Semi-classique: I
- Concentration Microlocale. {\em Commun. PDE.} 19:1535-1563, 1994.


\bibitem{CdVP-II} Y. Colin de Verdi\'ere  and B. Parisse. Equilibre Instable en R\'egime Semi-classique: II - conditions de Bohr-Sommerfeld.
{\em Ann. Inst. Henri Poincar\'e (Phyique th\'eorique)}. 61:347--367, 1994.




\bibitem{Datchev-Vasy:Gluing} K. Datchev and A. Vasy. Gluing semiclassical
  resolvent estimates via propagation of singularities. {\em
    Int. Math. Res. Notices},
to appear.

\bibitem{Datchev-Vasy:Trapping} K. Datchev and A. Vasy. Propagation
    through trapped sets and semiclassical resolvent estimates. {\em Annales de
  l'Institut Fourier}, to appear.

\bibitem{DaVa-trapping-2} K. Datchev and A. Vasy. Semiclassical
    Resolvent Estimates at Trapped Sets. {\em Preprint, arxiv 1206.0768}, 2012.

\bibitem{EmTr05}
M.~Embree and L.N. Trefethen.
\newblock {\em Spectra and pseudospectra, the behaviour of non-normal matrices
  and operators}.
\newblock Princeton University Press, 2005.

\bibitem{Fenichel} N. Fenichel. Persistence and smoothness of
    invariant manifolds for flow. {\em Indiana Univ. Math. J.} 21 No. 3
  (1972), 193--226.

\bibitem{Hirsch-Pugh-Shub} M.W. Hirsch, C.C. Pugh and M. Shub,
 \emph{Invariant manifolds} Lecture Notes in Mathematics,
  Vol. 583. Springer-Verlag, Berlin-New York, 1977.


\bibitem{Hit03} M. Hitrik, Eigenfrequencies and expansions for damped wave equations. 
{\em Methods Appl. Anal.} 10 (2003), no. 4, 543--564. 



\bibitem{Hormander:enseignement} L.~H{\"o}rmander. On the existence and the regularity of solutions of
  linear pseudo-differential equations. {\em Enseignement Math.} (2) \textbf{17}
  (1971), 99--163.

\bibitem{Leb93}
G.~Lebeau, {\em \'Equation des ondes amorties}, volume~19 of {\em Algebraic and
  geometric methods in mathematical physics}.
\newblock Kluwer Acad. Publ., Dordrecht, 1993.

\bibitem{Non11}
S.~Nonnenmacher.
\newblock Spectral theory of damped quantum systems.
\newblock {\em Preprint, arxiv:1109.0930}, 2011.

\bibitem{NoZw09_1}
S.~Nonnenmacher and M.~Zworski, Quantum decay rates in chaotic scattering.
\newblock {\em Acta Math.}, 203(2):149--233, 2009.

\bibitem{RaTa75}
J.~Rauch and M.~Taylor.
\newblock Decay of solutions to nondissipative hyperbolic systems on compact
  manifolds.
\newblock {\em Comm. in Pure and Appl. Anal.}, 28:501--523, 1975.

\bibitem{Sch10}
E.~Schenck.
\newblock Energy decay for the damped wave equation under a pressure condition.
\newblock {\em Comm. Math. Phys.}, 300(2):375--410, 2010.

\bibitem{Sch11}
E.~Schenck.
\newblock Exponential stabilization without geometric control.
\newblock {\em Math. Res. Lett.}, 18(2):379--388, 2011.

\bibitem{Sjo09}
J. Sj\"ostrand. Some results on nonselfadjoint operators: a survey. {\em Further progress in analysis},  45--74, World Sci. Publ., Hackensack, NJ, 2009. 

\bibitem{Sjo00}
J.~Sj\"ostrand.
\newblock Asymptotic distribution of eigenfrequencies for damped wave
  equations.
\newblock {\em Publ. Res. Inst. Math. Sci.}, 36(5):573--611, 2000.

\bibitem{SjZw91} J.~Sj{\"o}strand, Johannes and M.~Zworski,
\emph{Complex scaling and the distribution of scattering poles}, 
J. Amer. Math. Soc. 4 (1991), no. 4, 729--769. 


\bibitem{SjZw02} 
J. Sj\"ostrand and M. Zworski.  Quantum monodromy and semi-classical
trace formul\ae. {\em J. Math. Pures Appl.} (9) 81, 1--33, 2002.

\bibitem{Vasy-Dyatlov:Microlocal-Kerr}
A.~Vasy.
\newblock Microlocal analysis of asymptotically hyperbolic and {K}err-de
  {S}itter spaces.
\newblock {\em Preprint, arxiv:1012.4391}, 2010.
\newblock With an appendix by S. Dyatlov.

\bibitem{VaZw} A. Vasy and M. Zworski.  
Semiclassical estimates in asymptotically Euclidean scattering. 
{\em Comm. Math. Phys.} 212 (2000), no. 1, 205--217.

\bibitem{Wal75} P. Walters. {\em An introduction to ergodic theory},
  Springer, 1982.

\bibitem{WuZw} J. Wunsch and M. Zworski,
 \emph{Resolvent estimates for normally hyperbolic
    trapped sets}, Ann.\ Henri Poincar\'e,
  \textbf{12} (2011), 1349--1385.

\bibitem{Zworski:semiclassical} M. Zworski, \emph{Semiclassical
    analysis}, AMS Graduate Studies in Mathematics, 2012.
\end{thebibliography}
